\documentclass[a4paper,11pt]{amsart}

\newfont{\cyr}{wncyr10 scaled 1100}

\usepackage[left=2.7cm,right=2.7cm,top=3.5cm,bottom=3cm]{geometry}
\usepackage{amsthm,amssymb,amsmath,amsfonts,mathrsfs,amscd}
\usepackage[latin1]{inputenc}
\usepackage[all]{xy}
\usepackage{latexsym}
\usepackage{longtable}

\theoremstyle{plain}
\newtheorem{theorem}{Theorem}[section]
\newtheorem{corollary}[theorem]{Corollary}
\newtheorem{lemma}[theorem]{Lemma}
\newtheorem{proposition}[theorem]{Proposition}

\theoremstyle{definition}
\newtheorem{definition}[theorem]{Definition}

\theoremstyle{remark}

\newtheorem{remark}[theorem]{Remark}

\newcommand{\E}{\mathcal E}

\newcommand{\Q}{\mathbb{Q}}
\newcommand{\Z}{\mathbb{Z}}
\newcommand{\LL}{\mathbb{L}}
\newcommand{\F}{\mathbb{F}}

\newcommand{\C}{\mathbb{C}}

\newcommand{\X}{\mathcal{X}}
\newcommand{\J}{\mathcal{J}}
\newcommand{\JL}{\mathrm{JL}}

\newcommand{\Gal}{\operatorname{Gal}}

\newfont{\gotip}{eufb10 at 12pt}

\newcommand{\cO}{{\mathcal O}}

\newcommand{\lra}{\longrightarrow}

\newcommand{\Pic}{{\mathrm{Pic}}}

\newcommand{\R}{{\mathbb R}}
\newcommand{\M}{\mathcal M}

\newcommand{\m}{\mathfrak{m}}

\renewcommand{\l}{\ell}

\newcommand{\T}{\mathbb T}

\DeclareMathOperator{\Hom}{Hom}

\include{thebibliography}

\begin{document}

\title[Jochnowitz congruences on Shimura curves]{Heegner points and Jochnowitz congruences\\on Shimura curves}
\date{}
\author{Stefano Vigni}

\begin{abstract}
Given an elliptic curve $E$ over $\Q$, a suitable imaginary quadratic field $K$ and a quaternionic Hecke eigenform $g$ of weight $2$ obtained from $E$ by level raising such that the sign in the functional equation for $L_K(E,s)$ (respectively, $L_K(g,1)$) is $-1$ (respectively, $+1$), we prove a ``Jochnowitz congruence'' between the algebraic part of $L'_K(E,1)$ (expressed in terms of Heegner points on Shimura curves) and the algebraic part of $L_K(g,1)$. This establishes a relation between Zhang's formula of Gross--Zagier type for central derivatives of $L$-series and his formula of Gross type for special values. Our results extend to the context of Shimura curves attached to division quaternion algebras previous results of Bertolini and Darmon for Heegner points on classical modular curves.  
\end{abstract}

\address{Department of Mathematics, King's College London, Strand, London WC2R 2LS, United Kingdom}
\email{stefano.vigni@kcl.ac.uk}


\maketitle

\section{Introduction}

The goal of the present article is to use the theory of congruences between modular forms to relate, in the context of elliptic curves, Zhang's formula of Gross--Zagier type for central derivatives of $L$-functions and his formula for special values. This extends previous results of Bertolini and Darmon for Heegner points on modular curves (\cite{BD-ajm}) to the situation where one needs to work with Shimura curves attached to division quaternion algebras. More precisely, let $E$ be an elliptic curve over $\Q$ of conductor $N=MD$ where $D>1$ is a square-free product of an \emph{even} number of primes and $(M,D)=1$. By modularity, $E$ is associated with a normalised newform $f=f_E$ of weight $2$ for $\Gamma_0(N)$, whose $q$-expansion will be denoted by
\begin{equation} \label{q-expansion-eq}
f(q)=\sum_{n\geq1}a_n(f)q^n,\qquad a_n(f)\in\Z. 
\end{equation}
Let $K$ be an imaginary quadratic field, with ring of integers $\cO_K$ and discriminant coprime to $N$, in which the primes dividing $M$ (respectively, $D$) split (respectively, are inert); in other words, $K$ satisfies a \emph{modified Heegner hypothesis} relative to $E$. Let $X_0^D(M)$ be the (compact) Shimura curve over $\Q$ of discriminant $D$ and level $M$, and write $J_0^D(M)$ for its Jacobian variety. As recalled in \S \ref{JL-subsec}, the modularity of $E$ and the Jacquet--Langlands correspondence allow us to introduce a parametrisation
\[ \Pi_E:J_0^D(M)\longrightarrow E \]  
defined over $\Q$. The theory of complex multiplication produces a Heegner divisor class $P_K\in J_0^D(M)(K)$, and we define $\alpha_K:=\Pi_E(P_K)\in E(K)$ (see \S \ref{heegner-subsec}). The Heegner point $\alpha_K$ will play a key role in the formulation of our results, which now we briefly describe. 

After fixing in \S \ref{p-subsec} a suitable odd ``descent prime'' $p$, in \S \ref{l-subsec} we choose a Kolyvagin prime $\l$ relative to the data $(E,K,p)$. In particular, $\l$ is inert in $K$ and $p$ divides both $\l+1$ and $a_\l(f)$. As in \cite{BD-ajm}, the basic idea is to study Hecke congruences modulo $p$ between $f$ and modular forms of level $M\l$ (or rather, in our situation, between any quaternionic modular form of discriminant $D$ and level $M$ associated with $f$ via the Jacquet--Langlands correspondence and quaternionic modular forms of discriminant $D$ and level $M\l$). To do this, let $X_0^D(M\l)$ be the Shimura curve over $\Q$ of discriminant $D$ and level $M\l$, whose Jacobian we denote $J_0^D(M\l)$, and let $\T$ be the Hecke algebra of level $M\l$ acting on $J_0^D(M\l)$. Adopting the usual notation for Hecke operators and writing $-\epsilon$ for the sign in the functional equation for the $L$-function $L(E,s)$ of $E$, we introduce the maximal ideal $\m$ of $\T$ of residual characteristic $p$ defined by 
\[ \m:=\big\langle p;\quad T_r-a_r(f),\;r\nmid N\l;\quad U_q-a_q(f),\;q|M;\quad U_\l-\epsilon\big\rangle. \] 
Then we denote by $\T_\m$ the completion of $\T$ at $\m$ and by $I$ the kernel of the natural map $\T\rightarrow\T_\m$. As in \cite{BD-ajm}, we associate with $\m$ the quotient $J$ of $J_0^D(M\l)$ defined as
\[ J:=J_0^D(M\l)\big/IJ_0^D(M\l). \]
The abelian variety $J$ is the counterpart in our quaternionic setting of the abelian variety $J^{(\m)}$ introduced by Mazur in \cite[Ch. II, \S 10]{Ma}, the main difference being that here, as in \cite{BD-ajm}, the ideal $\m$ corresponds to an absolutely irreducible Galois representation modulo $p$ (thanks to the choice of $p$ made in \S \ref{p-subsec}) and hence is \emph{not} Eisenstein.

In \S \ref{J-subsec} we combine the Jacquet--Langlands correspondence with classical level raising results of Ribet to prove that $J$ is isogenous to $E^2\times J'$ where $J'$ is a non-zero abelian variety having purely toric reduction at $\l$. Moreover, the split or non-split nature of this reduction is controlled by $\epsilon$. The fact that $J'$ has purely toric reduction at $\l$ is important because it allows us to study the $\l$-adic points of $J'$ via the Tate--Mumford theory of non-archimedean uniformisation.

If $g$ is a Hecke eigenform of weight $2$ on $\Gamma_0^D(M\l)$ and $\cO_g$ is the ring generated by its Hecke eigenvalues then let $\phi_g:\T\rightarrow\cO_g$ be the (surjective) algebra homomorphism associated with $g$ and set $\m_g:=\phi_g(\m)$, which is a maximal ideal of $\cO_g$ (possibly equal to $\cO_g$ itself). The eigenform $g$ is said to be a form on $J$ (respectively, $J'$) if the abelian variety $A_g$ attached to $g$ by the Eichler--Shimura construction is a quotient of $J$ (respectively, $J'$).  

Now let $g$ be any form on $J'$ and notice that, since $\l$ in inert in $K$, the sign in the functional equation of $L_K(g,s)$ is $1$; in other words, in passing from level $M$ to level $M\l$ a sign of change occurs. In light of a formula of Gross (\cite{Gr}) later generalised by Daghigh and Zhang (\cite{Daghigh}, \cite{Zh2}), in \S \ref{algebraic-part-subsec} we define the \emph{algebraic part} $\LL_K(g,1)$ of the special value $L_K(g,1)$. More explicitly, $\LL_K(g,1)$ is introduced in terms of optimal embeddings of quadratic orders into Eichler orders of definite quaternion algebras and belongs to an $\cO_g$-module $\mathscr M$ that is locally free of rank $1$ at $\m_g$. Our definition is analogous to the one given in \cite[\S 4]{BD-ajm}, and the crucial property of this algebraic part is that $\LL_K(g,1)=0$ if and only if $L_K(g,1)=0$.

A simplified form of our main result, which extends \cite[Theorem 1.3]{BD-ajm} and is in fact a corollary of Theorem \ref{main-thm}, can be stated as follows. 

\begin{theorem} \label{main-intro-thm}
The image of $\alpha_K$ in $E(K_\l)/pE(K_\l)$ is nonzero if and only if
\[ \LL_K(g,1)\not\equiv0\pmod{\m_g} \]
for all forms $g$ on $J'$.
\end{theorem}

As will be apparent later, our strategy to prove Theorem \ref{main-intro-thm} is inspired by the arguments in \cite{BD-ajm}, and is actually an extension of these to a general quaternionic setting.

Now write $L_K(E,s)$ for the $L$-function of $E$ over $K$; the fact that $K$ satisfies our modified Heegner hypothesis ensures that the sign in the functional equation of $L_K(E,s)$ is $-1$ (see, e.g., \cite[Theorem 3.17]{Darmon} for a sketch of proof in the semistable case). By Zhang's formula of Gross--Zagier type (\cite[Theorem C]{Zh}), the Heegner point $\alpha_K$ encodes the central derivative $L'_K(E,1)$, hence Theorem \ref{main-intro-thm} can be viewed as providing a congruence modulo $\m$ between $L'_K(E,1)$ and $L_K(g,1)$. Congruences of this type (based on the sign-change phenomenon pointed out above) were first suggested by Jochnowitz (whence the title of the paper, see \cite{J}) and then studied and refined by, among others, Bertolini--Darmon and Vatsal (\cite{BD96}, \cite{BD05}, \cite{Vat2}). 

Some arithmetic consequences of Theorem \ref{main-intro-thm} are collected in \S \ref{consequences-subsec}. Here we remark that \cite[Theorem 4.2]{BD05} can be regarded as a generalisation of the main result of \cite{BD-ajm} to an Iwasawa-theoretic context. The strategy of proof of this result in \cite[\S 9]{BD05} follows closely the approach to Jochnowitz-type congruences proposed in \cite{Vat2}, avoiding the study of certain groups of connected components that appear both here and in \cite{BD-ajm}. Finally, we would like to point out that, in the setting of modular curves, Gross and Parson have established a link between the local $p$-divisibility of the Heegner point $\alpha_K$ at the prime $\l$ and the $p$-descent on a related abelian variety of level $N\l$ (see \cite{GP} for details).  

\section{Background on Shimura curves and Hecke algebras}

\subsection{Degeneracy maps and Hecke algebras} \label{degeneracy-subsec}

Let $B$ be the quaternion algebra over $\Q$ of discriminant $D$ and fix once and for all an isomorphism of $\R$-algebras
\begin{equation} \label{indefinite-iso-eq}
i_\infty:B\otimes_\Q\R\overset\cong\longrightarrow\M_2(\R),
\end{equation}
which exists because $B$ is indefinite (i.e., splits at the archimedean place of $\Q$). Let $\l$ be a prime number not dividing $N$, let $R(M\ell)\subset R(M)$ be Eichler orders of $B$ of level $M\ell$ and $M$, respectively, and let $\Gamma_0^D(M\ell)\subset\Gamma_0^D(M)$ be the corresponding groups of norm $1$ elements, so that
\[ X_0^D(M)=\Gamma_0^D(M)\backslash\mathcal H,\qquad X_0^D(M\ell)=\Gamma_0^D(M\ell)\backslash\mathcal H \]
as Riemann surfaces. There are two natural degeneracy maps
\[ \pi_1, \pi_2:X_0^D(M\ell)\longrightarrow X_0^D(M) \]
induced by the identity and the multiplication by $\omega_\ell$ on $\mathcal H$, respectively, where $\omega_\ell\in R(M\ell)$ has reduced norm $\ell$ (such an element normalises $\Gamma_0^D(M\ell)$). By covariant and contravariant functoriality, these degeneracy maps induce maps
\[ \pi_{1,*}, \pi_{2,*}:J_0^D(M\ell)\longrightarrow J_0^D(M),\qquad\pi_1^*, \pi_2^*:J_0^D(M)\longrightarrow J_0^D(M\ell) \]
between Jacobian varieties.

For any integer $S\geq1$ coprime to $D$ write $\T(S)$ for the Hecke algebra of level $S$, i.e., the subring of the endomorphism ring of $J_0^D(S)$ generated over $\Z$ by the Hecke operators $T_q$ with $q\nmid SD$ and $U_q$ with $q|S$. Then the degeneracy maps $\pi_1$ and $\pi_2$ satisfy the relations
\[ \pi_{1,*}\circ\pi_1^*=\pi_{2,*}\circ\pi_2^*=\ell+1,\qquad\pi_{2,*}\circ\pi_1^*=\pi_{1,*}\circ\pi_2^*=T_\ell. \]
Define an acton of $\T:=\T(M\ell)$ on $J_0^D(M)^2$ by letting the Hecke operators $T_q$ and $U_q$ for $q\not=\ell$ act diagonally and letting $U_\ell$ act by left multiplication by the matrix $\bigl(\begin{smallmatrix}T_\ell&\ell\\-1&0\end{smallmatrix}\bigr)$. Then let
\begin{equation} \label{pi-*-eq}
\pi^*:=\pi_1^*\oplus\pi_2^*:J_0^D(M)^2\longrightarrow J_0^D(M\ell),\qquad\pi_*:=(\pi_{1,*},\pi_{2,*}):J_0^D(M\ell)\longrightarrow J_0^D(M)^2 
\end{equation}
and set
\begin{equation} \label{tilde-pi-eq}
\tilde\pi_*:=\begin{pmatrix}1&-T_\ell\\0&1\end{pmatrix}\circ\pi_*. 
\end{equation}
The following result is proved as \cite[Lemma 2.1]{BD-ajm} (for a precise reference in the case of our interest, see \cite[p. 93]{Helm}).
\begin{lemma} \label{compatibility-lemma}
The maps $\pi^*$ and $\tilde\pi_*$ are compatible with the actions of $\T$ on $J_0^D(M\ell)$ and $J_0^D(M)^2$ defined above.
\end{lemma}

\subsection{Jacquet--Langlands and modularity} \label{JL-subsec}

Let $S_2(\Gamma_0^D(M))$ denote the $\C$-vector space of modular forms of weight $2$ on $\Gamma_0^D(M)$ (see \cite[\S 4.2]{Darmon}) and recall that $N=MD$. The Jacquet--Langlands correspondence (see, e.g., \cite[\S 4.5]{Darmon}, \cite[\S 1.4]{Helm-thesis}) gives a non-canonical isomorphism between $S_2^{\text{$D$-new}}(\Gamma_0(N))$ and $S_2(\Gamma_0^D(M))$ induced by a natural isomorphism at the level of Hecke algebras. Fix a form $f^\JL\in S_2(\Gamma_0^D(M))$ associated with $f$ via the Jacquet--Langlands correspondence; then $f^\JL$ is uniquely determined up to multiplication by elements in $\C^\times$. With notation as in \eqref{q-expansion-eq}, the form $f^\JL$ has the same Hecke eigenvalues as $f$ outside $D$, i.e.
\[ T_r(f^\JL)=a_r(f)f^\JL,\qquad U_q(f^\JL)=a_q(f)f^\JL \]
for all primes $r\nmid N$ and all primes $q|M$. 

The Jacquet--Langlands correspondence allows us to introduce a modular parametrisation of the elliptic curve $E$ by the Shimura curve $X_0^D(M)$. Following \cite{Zh} and \cite{Zh2}, let the \emph{Hodge class} be the unique $\xi_M\in\Pic(X_0^D(M))\otimes\Q$ of degree $1$ on which the Hecke operators at primes not dividing $M$ act as multiplication by their degree (see \cite[p. 30]{Zh} for an explicit expression of $\xi_M$ and \cite[\S 3.5]{CV} for a detailed exposition). Then one can define a map
\[ X_0^D(M)\lra J_0^D(M)\otimes\Q \]
by sending a point $x\in X_0^D(M)$ to the class $[x]-\xi_M$. Multiplying this map by a suitable integer $m\gg0$ gives a finite embedding
\begin{equation} \label{param-D-eq}
\iota_M:X_0^D(M)\lra J_0^D(M)
\end{equation}
defined over $\Q$ (cf. \cite[\S 3.5]{CV}), which we fix once and for all. Now choose a parametrisation
\[ \Pi_E:J_0^D(M)\longrightarrow E \]
defined over $\Q$, whose existence is guaranteed by the modularity of $E$, the Jacquet--Langlands correspondence and Faltings's isogeny theorem (see \cite[\S 4.6]{Darmon}, \cite[\S 3.4.4]{Zh}, \cite[\S 5]{Zh3}). Finally, set
\[ \pi_E:=\Pi_E\circ\iota_M:X_0^D(M)\longrightarrow E, \]
which is a surjective morphism defined over $\Q$. The map $\pi_E$ induces two maps
\begin{equation} \label{functoriality-maps-eq}
\pi_{E,*}=\Pi_E:J_0^D(M)\longrightarrow E,\qquad\pi_E^*:E\longrightarrow J_0^D(M) 
\end{equation}
by Albanese (i.e., covariant) and Picard (i.e., contravariant) functoriality, respectively. At the cost of replacing $E$ with an isogenous curve, from now on we shall always assume that $E$ is a strong Weil curve and that $\pi_E$ is a strong Weil parametrisation, in the sense that $\pi_{E,*}$ has connected kernel (or, equivalently, that $\pi_E^*$ is injective). 

\subsection{Enhanced QM surfaces and character groups of Jacobians} \label{enhanced-subsec}

In what follows we shall deal with two different quaternion algebras: $B$ is the \emph{indefinite} quaternion algebra over $\Q$ of discriminant $D$ introduced in \S \ref{degeneracy-subsec}, while $\mathcal B$ is the \emph{definite} quaternion algebra over $\Q$ of discriminant $D\ell$. The interplay between $B$ and $\mathcal B$ lies at the core of our subsequent considerations.

We refer the reader to \cite[\S 1]{BD98} for the notion of \emph{oriented} Eichler order. Let $R_1,\dots,R_t$ be representatives for the conjugacy classes of oriented Eichler orders of level $M$ in $\mathcal B$; we denote their classes $[R_i]$ and set
\[ \mathcal E:=\bigl\{[R_1],\dots,[R_t]\bigr\}. \]
Let $\mathscr M$ stand for the free abelian group over $\mathcal E$, i.e. $\mathscr M:=\Z[\mathcal E]$, the set of all formal $\Z$-linear combinations of the $[R_i]$. We want to describe a geometric interpretation of $\mathscr M$ in terms of abelian surfaces with quaternionic multiplication (QM surfaces, for short), whose definition is recalled, e.g., in \cite[\S 4]{BD98}. With a terminology analogous to that of \cite{Ri}, we give

\begin{definition} \label{enhanced-QM-dfn}
An \emph{enhanced QM surface} with $M$-level structure over a field $k$ is a pair $(A,C)$ where $A$ is a QM surface over $k$ and $C\subset A$ is a $k$-rational subgroup of order $M$ whose points over the algebraic closure of $k$ form a cyclic group.
\end{definition}

When $k$ is $\bar\F_\ell$ (or a finite extension of $\F_\ell$) we say that $(A,C)$ is an enhanced QM surface in characteristic $\ell$.

There is an evident notion of isomorphism between such enhanced QM surfaces. As in \cite[\S 5]{BD98}, write $\underline{\text{End}}(A)$ for the endomorphism ring of the pair $(A,C)$. It is a basic fact that if $(A,C)$ is an enhanced QM surface with $M$-level structure in characteristic $\ell$ then $\underline{\text{End}}(A)$ is (isomorphic to) either an order in an imaginary quadratic field or an Eichler order of level $M$ in $\mathcal B$ (cf. \cite[Theorem 4.2.1]{Ro}).

\begin{definition} \label{supersingular-enhanced-dfn}
An enhanced QM surface $(A,C)$ with $M$-level structure in characteristic $\ell$ is \emph{supersingular} if $\underline{\text{End}}(A)$ is an Eichler order of level $M$ in $\mathcal B$.
\end{definition}

The following result is a generalisation of \cite[Proposition 3.3]{Ri}.

\begin{proposition} \label{supersing-eichler-pro}
The set of isomorphism classes of supersingular QM surfaces in characteristic $\ell$ is in bijection with the set of oriented Eichler orders of level $M$ in $\mathcal B$.
\end{proposition}
\begin{proof} The bijection sends the class $[(A,C)]$ to the Eichler order $\underline{\text{End}}(A)$, which is oriented as in \cite[Remark 1.2.4]{Ro}. For details, see \cite{Ro}. \end{proof}

Write $\X_0^D(M)$ for Drinfeld's model of $X_0^D(M)$ over $\Z$ (\cite[\S 4.1]{Ro}) and let $\X:=\X_0^D(M)\times_\Z\Z_\ell$ be its base change to $\Z_\ell$. Since $\ell\nmid N$, the special fibre $\tilde\X$ of $\X$ is non-singular (\cite[\S 4.2]{Ro}) and its points over $\bar\F_\ell$ correspond to classes of enhanced QM surfaces in characteristic $\ell$. We say that $[(A,C)]\in\tilde\X_{/\bar\F_\ell}$ is \emph{supersingular} if $(A,C)$ is supersingular according to Definition \ref{supersingular-enhanced-dfn}.

Let $\Sigma$ be the set of supersingular points of $\tilde\X_{/\bar\F_\ell}$. We immediately get

\begin{proposition} \label{sigma-e-prop}
The sets $\Sigma$ and $\mathcal E$ are in bijection.
\end{proposition}

\begin{proof} A reinterpretation of Proposition \ref{supersing-eichler-pro}. \end{proof}

Now let $\J_0^D(M\ell)$ be the N\'eron model of $J_0^D(M\ell)$ over $\Z_\ell$ and let $\J_0^D(M\ell)^0$ be its identity component. The special fibre of $\J_0^D(M\ell)^0$ at $\ell$ is an extension of the abelian variety $J_0^D(M)^2$ over $\F_\ell$ by a torus $T$. The character group $X^*(T)$ of $T$ is a free abelian group of finite rank that inherits an action of the Hecke algebra $\T$. In fact, classical results of Grothendieck and Raynaud ensure that $X^*(T)$ is  isomorphic to the group $\Z[\Sigma]^0$ of degree zero divisors on $\Sigma$. By Proposition \ref{sigma-e-prop}, the free abelian group $\Z[\Sigma]$ is naturally identified with $\mathscr M$, hence an element in $X^*(T)$ will sometimes be viewed as a $\Z$-linear combination $\sum_jn_j[R_j]$ with $\sum_jn_j=0$.

Finally, note that $\mathscr M$ is equipped with a natural positive-definite scalar product
\begin{equation} \label{scalar-product-eq} 
\big\langle[R_i],[R_j]\big\rangle:=\frac{1}{2}\delta_{ij}\#R_j^\times.
\end{equation}
Equivalently, $\big\langle[R_i],[R_j]\big\rangle$ is the number of isomorphisms between $R_i$ and $R_j$. 

\subsection{Component groups and multiplicity one} \label{components-subsec}

Let $F$ be a finite extension of $\Q_\ell$ with ring of integers $\cO_F$ and ramification index $e=e_F$, and let $\J_0^D(M\ell)_F$ denote the N\'eron model of $J_0^D(M\ell)$ over $\cO_F$. The group $\Phi\bigl(J_0^D(M\ell)_{/F}\bigr)$ of connected components of the special fibre of $\J_0^D(M\ell)_F$ can be described canonically as the cokernel of the composition
\[ X^*(T)\overset{e}\longrightarrow X^*(T)\longrightarrow X^*(T)^\vee, \]
where the first map is multiplication by $e$ and the second one is the natural inclusion of $X^*(T)$ into $X^*(T)^\vee:=\Hom(X^*(T),\Z)$ arising from pairing \eqref{scalar-product-eq}. It follows that there is a short exact sequence
\begin{equation} \label{component-groups-eq}
0\longrightarrow X^*(T)\otimes(\Z/e\Z)\longrightarrow\Phi\bigl(J_0^D(M\ell)_{/F}\bigr)\longrightarrow\Phi\bigl(J_0^D(M\ell)_{/\Q_\ell}\bigr)\longrightarrow0, \end{equation}
the last map being induced from the norm if the extension $F/\Q_\ell$ is totally ramified. The Hecke module $\Phi(J_0^D(M\ell)_{/\Q_\ell})$ is Eisenstein (see \cite[p. 95]{Helm}), that is, its support consists exclusively of maximal ideals $\m$ of $\T$ that are Eisenstein in the sense that their associated residual Galois representations $\bar\rho_\m$ fail to be absolutely irreducible. In particular, we obtain

\begin{proposition} \label{non-eisenstein-prop}
The completion of $\Phi(J_0^D(M\l)_{/F})$ at a non-Eisenstein maximal ideal $\m$ of $\T$ is isomorphic as a Hecke module to the completion of $X^*(T)\otimes(\Z/e\Z)$ at $\m$.
\end{proposition}  

\begin{proof} Passing to the $\m$-adic completions in \eqref{component-groups-eq} produces a short exact sequence of Hecke modules, and the claim is a consequence of the fact that the completion of $\Phi\bigl(J_0^D(M\ell)_{/\Q_\ell}\bigr)$ at the non-Eisenstein ideal $\m$ is trivial. \end{proof}   

Now we recall the notion of \emph{controllability} as introduced by Helm (\cite[Definition 6.4]{Helm}).

\begin{definition} \label{controllable-dfn}
Let $\m$ be a non-Eisenstein maximal ideal of $\T$ of residual characteristic $p$. Then $\m$ is \emph{controllable at $\l$} if one of the following conditions holds: 
\begin{enumerate}
\item $\bar\rho_\m$ is not finite at $\ell$;
\item $\bar\rho_\m$ is unramified at $\ell$, $\ell\not=p$ and $\bar\rho_\m(\mathrm{Frob}_\l)$ is not a scalar;
\item $\l=p$ and $p\not=2$;
\item $\l=p=2$ and the restriction of $\bar\rho_\m$ to a decomposition group at $2$ is not contained in the scalar matrices.
\end{enumerate}
\end{definition}

The following mutiplicity one result is the counterpart of \cite[Theorem 6.4]{Ri}.

\begin{theorem}[Helm] \label{1-dimensional-thm}
Suppose that $\m$ is a maximal ideal of $\T$ that is controllable at $\l$. Then the quotient $X^*(T)/\m X^*(T)$ is a one-dimensional vector space over $\T/\m$.
\end{theorem}

\begin{proof} This is \cite[Lemma 6.5]{Helm}.  \end{proof}

\begin{remark}
Let $\T_\m$ be the completion of $\T$ at the ideal $\m$. By Nakayama's lemma, Theorem \ref{1-dimensional-thm} implies that the completion $X^*(T)\otimes_\T\T_\m$ is free of rank one over $\T_\m$.
\end{remark}

\section{A distinguished abelian variety}

\subsection{The auxiliary prime $p$} \label{p-subsec}

Let us fix once and for all an auxiliary prime number $p$ (the ``descent prime'') such that
\begin{enumerate}
\item the mod $p$ Galois representation $\bar\rho_{E,p}$ attached to $E$ is absolutely irreducible;
\item $p$ does not divide $6N$ or the degree $\deg(\pi_E)$ of the modular parametrisation $\pi_E$.
\end{enumerate}
Serre's ``open image theorem'' (\cite{Se}) guarantees that all but finitely many primes $p$ satisfy these two conditions. 

\subsection{Choice of $\l$ and level raising} \label{l-subsec}

Let $p$ be as in \S \ref{p-subsec}, let $E[p]$ be the group of $p$-torsion points of $E$ and write $K(E[p])$ for the smallest extension of $K$ containing the coordinates of these points. Moreover, let $\delta_K$ be the discriminant of $K$. Recall from \cite[p. 261]{BD-ajm} that a prime $\l$ is called a \emph{Kolyvagin prime} (relative to $E$, $K$ and $p$) if 
\begin{enumerate}
\item $\l\nmid N\delta_Kp$ (so that $\l$ is unramified in $K(E[p])$);
\item $\text{Frob}_\l\in\Gal(K(E[p])/\Q)$ belongs to the conjugacy class of complex conjugation.
\end{enumerate} 
In particular, $\l$ is inert in $K$ and $p$ divides both $\l+1$ and $a_\l(f)$. By \v{C}ebotarev's density theorem, there are infinitely many Kolyvagin primes relative to $(E,K,p)$. Note that every Kolyvagin prime $\l$ satisfies the claims of \cite[Lemma 7.1]{Helm} (see the proof of \cite[Lemma 7.1]{Helm}).

Fix a Kolyvagin prime $\l$ relative to $(E,K,p)$ and, as before, write $\T$ for the Hecke algebra $\T(M\l)$ of level $M\l$. As in \cite{BD-ajm}, our goal is to study certain modular forms of level $M\l$ that are congruent to $f$ (or, rather, to $f^\JL$) modulo $p$. Recall from the introduction that $-\epsilon$ is the sign in the functional equation for $L(E,s)$ and let $\m=\m_f$ be the maximal ideal of $\T$ of residual characteristic $p$ defined by
\[ \m:=\big\langle p;\quad T_r-a_r(f),\;r\nmid N\l;\quad U_q-a_q(f),\;q|M;\quad U_\l-\epsilon\big\rangle. \] 
If $M$ is a $\T$-module then denote $M_\m$ the completion of $M$ at $\m$, so that
\[ M_\m:=\varprojlim_nM/\m^nM. \]
The ideal $\m$ is said to be in the \emph{support} of $M$ if $M_\m\not=0$.

Since $\T$ is a finitely generated $\Z$-module, $\T_\m$ is a direct factor of the semilocal ring $\T\otimes\Z_p$; write
\begin{equation} \label{splitting-T-eq}
\T\otimes\Z_p=\T_\m\times\T',
\end{equation}
where $(\T')_\m=0$. Furthermore, let $I$ be the kernel of the natural map $\T\rightarrow\T_\m$.

\begin{remark}
Using Nakayama's lemma, it can be checked that $\m$ belongs to the support of $M$ if and only if the localisation of $M$ at $\m$ is not trivial.
\end{remark}

In Definition \ref{controllable-dfn} we explained the notion of controllability for (non-Eisenstein) maximal ideals of $\T$; here we record the following

\begin{lemma} \label{controllable-lem}
The ideal $\m$ is controllable at $\l$.
\end{lemma}

\begin{proof} First of all, the Galois representation $\bar\rho_\m$ is isomorphic to $\bar\rho_{E,p}$, so it is absolutely irreducible by assumption (1) made on $p$ in \S \ref{p-subsec}. Therefore $\m$ is non-Eisenstein. On the other hand, $\l$ is a Kolyvagin prime, which means that $\text{Frob}_\l\in\Gal(K(E[p])/\Q)$ belongs to the conjugacy class of complex conjugation. Since complex conjugation acts on $E[p]$ with eigenvalues $\pm1$, it follows that $\bar\rho_\m(\text{Frob}_\l)$ is not a scalar. \end{proof} 

Finally, multiply the Hodge classes $\xi_M$ and $\xi_{M\l}$ by the same integer $m\gg0$ in order to define maps $\iota_M$ and $\iota_{M\l}$ as in \eqref{param-D-eq}, so that the natural square
\begin{equation} \label{iota-square-eq}
\xymatrix{X_0^D(M\l)\ar[r]^-{\iota_{M\l}}\ar[d]^-{\pi_1}&J_0^D(M\l)\ar[d]^-{\pi_{1,*}}\\X_0^D(M)\ar[r]^-{\iota_M}&J_0^D(M)}
\end{equation} 
is commutative. This is a consequence of the functoriality of the Hodge class, as explained in \cite[\S 3.5]{CV}.   

\subsection{The modular form $f_\l$} \label{f-l-subsec}

The cusp form $f\in S_2(\Gamma_0(N))$ associated with $E$ is not an eigenform for the Hecke algebra of level $N\ell$, because it fails to be an eigenform for the Hecke operator at $\l$. As in \cite[p. 267]{BD-ajm}, choose a root $\beta$ of the polynomial $X^2-a_\l(f)X+\l$ and define the (classical) modular form
\[ f_\l(z):=f(z)-\l/\beta f(\l z) \]
with coefficients in the imaginary quadratic order $\Z[\beta]$. Then $f_\l$ is an eigenform of level $N\l$ that is in the same old-class as $f$ and has eigenvalue $\beta$ at $\l$. Note that, since $\l$ is a Kolyvagin prime, $\text{Frob}_\l$ acts on $E[p]$ with eigenvalues $1$ and $-1$, that is
\[ X^2-a_\l(f)X+\l\equiv(X-1)\cdot(X+1)\pmod{p}. \]
Since $p$ is odd, it splits in $\Z[\beta]$ and is equal to the product $(p,\beta-1)\cdot(p,\beta+1)$.

Fix a form $f_\l^\JL\in S_2(\Gamma_0^D(M\l))$ corresponding, as in \S \ref{JL-subsec}, to $f_\l$ under the Jacquet--Langlands correspondence and let $\m_{f^\JL}$ be the ideal $(p,\beta-\epsilon)$ of $\Z[\beta]$. Then the maximal ideal $\m$ of \S \ref{l-subsec} is the inverse image of $\m_{f^\JL}$ under the homomorphism $\T\rightarrow\Z[\beta]$ determined by $f_\l^\JL$.  

\subsection{The abelian variety $J$} \label{J-subsec}

Let $I\subset\T$ be as in \S \ref{l-subsec} and consider the abelian variety
\[ J:=J_0^D(M\l)\big/IJ_0^D(M\l). \]
If $g$ is a Hecke eigenform of weight $2$ on $\Gamma_0^D(M\l)$, so that $g$ corresponds to a one-dimensional Hecke-stable subspace of $S_2^{\text{$D$-new}}(\Gamma_0(N\l))$ by Jacquet--Langlands, and $\cO_g$ is the ring generated by its Hecke eigenvalues then let $\phi_g:\T\rightarrow\cO_g$ be the (surjective) algebra homomorphism associated with $g$ (cf. \cite[Theorem 1.2]{BD96}), whose kernel we denote $I_g$. Observe that, with notation as in \S \ref{f-l-subsec}, $\Z[\beta]=\cO_{f_\l^\JL}\cong\T/I_{f_\l^\JL}$. 

\begin{definition} \label{form-on-J-dfn}
The form $g$ is a \emph{form on $J$} if the following equivalent conditions hold:
\begin{enumerate}
\item the abelian variety $A_g:=J_0^D(M\l)\big/I_g J_0^D(M\l)$ associated with $g$ is a quotient of $J$;
\item $I\subset I_g$;
\item the ideal $\m_g:=\phi_g(\m)$ is a proper maximal ideal of $\cO_g$ and
\[ a_n(f^\JL_\ell)\pmod{\m_{f^\JL}}=a_n(g)\pmod{\m_g} \]
for all integers $n\geq1$ coprime to $D$.
\end{enumerate}
\end{definition}

The absolute Galois group $G_\Q$ of $\Q$ acts on the Hecke eigenforms of weight $2$ on $\Gamma_0^D(M\l)$ and the abelian variety $A_g$ in Definition \ref{form-on-J-dfn} depends only on the Galois orbit $[g]$ of $g$. Let $t$ be the dimension of $J_0^D(M\l)$, let $\mathcal F=\bigl\{g_1,\dots,g_t\}$ be a basis of $S_2(\Gamma_0^D(M\l))$ consisting of eigenforms for $\T$ and assume (at the cost of renumbering) that $\{g_1,\dots,g_m\}$ is a set of representatives for the set of orbits of $\mathcal F$ under the action of $G_\Q$. Let us fix an isogeny
\begin{equation} \label{isogeny-eq}
J_0^D(M\l)\sim\prod_{i=1}^mA_{g_i}. 
\end{equation}
Assuming also that $\{g_1,\dots,g_d\}$ is a set of representatives for the distinct Galois orbits of eigenforms on $J$ according to Definition \ref{form-on-J-dfn}, it follows from \eqref{isogeny-eq} that there is an isogeny
\[ J\sim\prod_{i=1}^dA_{g_i}. \]
Since the prime $p$ does not divide the degree of the modular parametrisation $\pi_E$ (cf. \S \ref{p-subsec}), there is only one oldform (up to the Galois action) that is congruent to $f_\ell^\JL$ modulo $p$ (that is, the Hecke eigenvalues are congruent modulo $p$), namely $f_\l^\JL$ itself. 

Finally, the abelian variety $A_{f_\l^\JL}=J_0^D(M\l)\big/I_{f_\l^\JL}J_0^D(M\l)$ is isogenous to $E\times E$, hence there is an isogeny 
\begin{equation} \label{J-isogeny-eq}
J\sim E^2\times\prod_{[g]}A_g
\end{equation}
where the product is taken over the Galois orbits of the eigenforms on $J$ that are new at $\l$, in the sense that the classical cusp forms associated with them by the Jacquet--Langlands correspondence are new at $\l$.

As in \cite{BD-ajm}, we give a description of such an isogeny. Recall from \eqref{functoriality-maps-eq} that $\pi_E^*:E\rightarrow J_0^D(M)$ is the map induced from $\pi_E$ by Picard functoriality. The map $\pi_E^*$ induces in turn a morphism $E^2\rightarrow J_0^D(M)^2$, also denoted $\pi_E^*$. Consider the composition
\begin{equation} \label{phi-E-eq}
\varphi_E:E^2 \xrightarrow{\pi_E^*}J_0^D(M)^2\overset{\pi^*}\longrightarrow J_0^D(M\l)\overset{\pi_J}\longrightarrow J,
\end{equation}
where $\pi^*$ is as in \eqref{pi-*-eq} and $\pi_J$ is the canonical projection. Thanks to Lemma \ref{compatibility-lemma}, the map $\varphi_E$ is $\T$-equivariant with respect to the action  of $\T$ on $E^2$ defined by letting the operators $T_r$ for $r\nmid N\l$ and $U_r$ for $r|M$ act by multiplication by $a_r(f)$ and letting $U_\l$ act by left mutiplication by the matrix $\bigl(\begin{smallmatrix}a_\l(f)&\l\\-1&0\end{smallmatrix}\bigr)$. 

\subsection{The $\l$-new subvariety and its reduction}

Recall that an abelian variety $A$ over $\Q$ is said to have \emph{purely toric} (or \emph{purely multiplicative}) reduction at a prime $q$ if the identity component of the special fibre of the N\'eron model of $A$ over $\Z_q$ is a torus. Moreover, if this torus is a \emph{split} (respectively, \emph{non-split}) torus over $\mathbb F_q$ then $A$ is said to have purely \emph{split} (respectively, \emph{non-split}) toric reduction at $q$. Let $J_0^D(M\l)^{\text{$\l$-new}}$ be the $\l$-new subvariety of $J_0^D(M\l)$, that is, the kernel of the map $\pi_*$ introduced in \eqref{pi-*-eq}. 

The following result is well known to the experts.

\begin{proposition} \label{toric-maximal}
The abelian variety $J_0^D(M\l)^{\text{$\l$-new}}$ has purely toric reduction at $\l$; it is the maximal toric subvariety of $J_0^D(M\l)$ at $\l$.
\end{proposition}

\begin{proof}[Sketch of proof] Write $A$ for the $\l$-new quotient of $J_0^D(M\l)$, i.e., the cokernel of the map $\pi^*$ defined in \eqref{pi-*-eq}. First of all, since $J_0^D(M\l)^{\text{$\l$-new}}$ and $A$ are isogenous, to show that $J_0^D(M\l)^{\text{$\l$-new}}$ has purely toric reduction at $\l$ it suffices to show that the same is true of $A$. On the other hand, there is a short exact sequence
\[ 0\longrightarrow T\longrightarrow\J_0^D(M\l)^0_{/\mathbb F_\l}\longrightarrow\J^D_0(M)_{/\mathbb F_\l}\times\J^D_0(M)_{/\mathbb F_\l}\longrightarrow0 \]
where $\J_0^D(M\l)^0_{/\mathbb F_\l}$ denotes the identity component of the special fibre of the N\'eron model of $J_0^D(M\l)$ over $\Z_\l$, $T$ is the maximal torus in this fibre and $\J^D_0(M)_{/\mathbb F_\l}$ is the special fibre of the N\'eron model of $J_0^D(M)$ over $\Z_\l$ (this is a consequence of the properties of Drinfeld's model of $X_0^D(M\l)$ over $\Z_\l$; see, e.g., \cite[\S 4.3]{Ro} for details). Then the arguments in the proof of \cite[Proposition 1]{L}, where an analogous result for modular Jacobians is shown, apply \emph{mutatis mutandis} to our setting as well. \end{proof}

\subsection{The abelian variety $J'$}

Write $J'$ for the image of $J_0^D(M\l)^{\text{$\l$-new}}$ in $J$ and let $\varphi':J'\hookrightarrow J$ be the natural inclusion, then define an isogeny $\varphi$ as
\begin{equation} \label{phi-eq}
\varphi:=\varphi_E+\varphi':E^2\times J'\longrightarrow J. 
\end{equation}

The following result is the analogue of \cite[Proposition 3.1]{BD-ajm}. Note that, in our quaternionic context, it is necessary to invoke general level raising results obtained by Diamond and Taylor in \cite{DT} and not just those originally proved by Ribet in \cite{Ri84}. 

\begin{proposition} \label{mult-red-prop}
The abelian variety $J'$ is not trivial and has purely toric reduction at $\l$. Furthermore, this reduction is split if $\epsilon=1$ and is non-split if $\epsilon=-1$.
\end{proposition}

\begin{proof} By \eqref{J-isogeny-eq} and Poincar\'e's complete reducibility theorem, there is an isogeny
\[ J'\sim\prod_{[g]}A_g, \]
where the product is over the Galois orbits of the eigenforms on $J$ which are new at $\l$. These are weight $2$ eigenforms $g$ on $\Gamma_0^D(M\l)$ that are new at $\l$, satisfy $\m_g\not=\cO_g$ and such that the congruence
\begin{equation} \label{cong-hecke-eq} 
a_n(f^\JL_\ell)\pmod{\m_{f^\JL}}=a_n(g)\pmod{\m_g} 
\end{equation}
holds for all $n\geq1$ coprime to $D$ (cf. Definition \ref{form-on-J-dfn}). Thanks to our choice of $\l$, level raising results of Diamond--Taylor (\cite{DT}) ensure that there exists a form $g$ that is new at $\l$ and satisfies \eqref{cong-hecke-eq} (see \cite[Theorem 1.6.4]{Helm-thesis} and \cite[Lemma 7.1]{Helm} for precise statements). It follows that the abelian variety $J'$ is not trivial. As for the part about reduction, $J'$ is (isomorphic to) a quotient of $J_0^D(M\l)^{\text{$\l$-new}}$, hence Proposition \ref{toric-maximal} implies that $J'$ has purely toric reduction at $\l$. Finally, with $g$ being any form as above, it is known that $a_\l(g)=1$ (respectively, $a_\l(g)=-1$) if and only if $A_g$ has split (respectively, non-split) multiplicative reduction at $\l$. But $a_\l(g)\equiv\epsilon\pmod{\m_g}$ and $p\not=2$, and the proposition is proved. \end{proof}

Let $K_\ell$ be the completion of $K$ at $\l\cO_K$. Proposition \ref{mult-red-prop} allows us to apply to $J'$ the Tate--Morikawa--Mumford theory of non-archimedean ($\l$-adic) uniformisation of abelian varieties with purely toric reduction (see \cite[Section III]{Ri76} for an exposition). In particular, we obtain

\begin{corollary} \label{mumford-coro}
\begin{enumerate}
\item Let $d$ be the dimension of $J'$. There is a short exact sequence
\[ 0\longrightarrow\Lambda\longrightarrow(K_\l^\times)^d\longrightarrow J'(K_\l)\longrightarrow0 \]
where $\Lambda$ is a full rank lattice in $(K_\l^\times)^d$.
\item Denote by $z\mapsto\bar z$ the action of complex conjugation on $K_\l$ and let $\Gal(K_\l/\Q_\l)=\langle\tau\rangle$ act on $K_\l^\times$ by $\tau(z):=\bar z^\epsilon$. Then the above exact sequence is $\Gal(K_\l/\Q_\l)$-equivariant.
\end{enumerate}
\end{corollary}

An analogous statement holds for any finite extension of $K_\l$. Now recall the map $\varphi_E$ defined in \eqref{phi-E-eq}.

\begin{lemma} \label{support-lemma}
The maximal ideal $\m$ is not in the support of the kernel of $\varphi_E$.
\end{lemma} 

\begin{proof} Thanks to our assumption that $\pi_E$ is a strong Weil parametrisation (\S \ref{JL-subsec}), the map $\pi_E^*:E^2\rightarrow J_0^D(M)^2$ is injective. On the other hand, by the results in \cite{DT} (see, in particular, \cite[\S 3]{DT}), the support in $\T$ of the kernel of $\pi^*$ consists entirely of Eisenstein maximal ideals. Finally, if $M$ is a submodule of $IJ_0^D(M\l)(\bar\Q)$ then the action of $\T\otimes\Z_p$ on $M\otimes\Z_p$ factors through the ring $\T'$ appearing in \eqref{splitting-T-eq}, hence the support of $M$ is disjoint from $\m$. Since $\m$ is not Eisenstein (cf. Lemma \ref{controllable-lem}), we are done. \end{proof}

As a general notation, if $M$ is a module on which complex conjugation in $G_\Q$ acts then write $M^+$ (respectively, $M^-$) for the submodule of $M$ on which this involution acts as the identity (respectively, as $-1$).

Write $\cO_\l$ for the ring of integers of $K_\l$. The next result studies the kernel of the isogeny $\varphi$ introduced in \eqref{phi-eq}.

\begin{lemma} \label{kernel-lemma}
Let $V$ be the kernel of $\varphi$. Then
\begin{enumerate}
\item the map $V\rightarrow E^2$ induce by projection onto the first factor is injective, so that $V_{/K_\l}$ extends to a finite flat group scheme over $\cO_\l$;
\item the map $V(K_\l)^\epsilon_\m\rightarrow E^2(K_\l)^\epsilon_\m$ is an isomorphism.
\end{enumerate}
\end{lemma}

\begin{proof} Using the map $\tilde\pi_*$ defined in \eqref{tilde-pi-eq} and the fact that $p\nmid\deg(\pi_E)$ by condition (2) in \S \ref{p-subsec}, one can proceed as in the proof of \cite[Lemma 3.4]{BD-ajm}. \end{proof}

\subsection{An interlude on ring class fields} \label{ring-class-subsec}

For later reference, we review some basic facts about ring class fields. Let $H=H_K$ be the Hilbert class field of $K$ and let $\mathcal L$ be the ring class field of $K$ of conductor $\l$, so that $\Gal(H/K)\cong\Pic(\cO_K)$ and if $\cO_\l$ is the order of $K$ of conductor $\l$ then $\Gal(\mathcal L/K)\cong\Pic(\cO_\l)$. The field $\mathcal L$ is a cyclic extension of $H$ of degree $(\l+1)/u$ where $u:=\#\cO_K^\times/2$. 

The Kolyvagin prime $\l$ is inert in $K$, hence it splits completely in $H/K$, by class field theory. Moreover, every prime of $H$ above $\l$ is totally ramified in $\mathcal L$. Write $\mathcal L_\l$ for the completion of $\mathcal L$ at any such prime above $\l$; then $\mathcal L_\l$ is a totally ramified cyclic extension of $K_\l$ of degree $(\l+1)/u$. Fix a generator $\sigma=\sigma_\l$ of $\Gal(\mathcal L_\l/K_\l)=\Gal(\mathcal L/H)$. The group $\Gal(\mathcal L_\l/\Q_\l)$ is isomorphic to a dihedral group of order $2(\l+1)/u$ and complex conjugation in $\Gal(K_\l/\Q_\l)$ acts on $\Gal(\mathcal L_\l/K_\l)$ by sending $\sigma$ to $\sigma^{-1}$.

\subsection{An analysis of component groups}

We introduce some more notations that will be used in the rest of the paper. In analogy with what done in \S \ref{components-subsec} for $J_0^D(M\l)$, if $F$ is a finite field extension of $K_\l$ with ring of integers $\cO_F$ then denote $\J_F$ (respectively, $\J'_F$) the N\'eron model of $J$ (respectively, $J'$) over $\cO_F$, and write $\Phi(J_{/F})$ (respectively, $\Phi(J'_{/F})$) for the component group of this N\'eron model. By a slight abuse of notation, we shall also write $J_0^D(M\l)(F)$ in place of $\J^D_0(M\l)(\cO_F)$, and similarly for $J$ and $J'$. In particular, $J^0(F)$ will denote the identity component in $\J(\cO_F)$ and so on. 

We need two more auxiliary results, which correspond to \cite[Lemmas 3.5 and 3.6]{BD-ajm}, on completions at $\m$ of component groups.

\begin{lemma} \label{Phi-iso-lemma}
The natural map
\[ \Phi\bigl(J_0^D(M\l)_{/F}\bigr)_{\!\m}\longrightarrow\Phi(J_{/F})_\m \]
is an isomorphism.
\end{lemma}

\begin{proof} Let $I^\perp:=\text{Ann}_\T(I)$ and $J^\perp:=J_0^D(M\l)/I^\perp$. The natural isogeny
\[ J_0^D(M\l)\longrightarrow J\times J^\perp \]
has a kernel that is contained in $IJ_0^D(M\l)$, whose support is then disjoint from $\m$ (cf. the proof of Lemma  \ref{support-lemma}). Therefore this isogeny induces an isomorphism
\[ \Phi\bigl(J_0^D(M\l)_{/F}\bigr)_{\!\m}\overset\cong\longrightarrow\Phi(J_{/F})_\m\times\Phi\bigl(J^\perp_{/F}\bigr)_{\!\m}. \] 
But the action of the algebra $\T\otimes\Z_p$ on $\Phi(J^\perp)\otimes\Z_p$ factors through the summand $\T'$ defined in \eqref{splitting-T-eq}, hence $\Phi(J^\perp)_\m=0$. \end{proof}

As in Lemma \ref{kernel-lemma}, let $V$ denote the kernel of the map $\varphi$ of \eqref{phi-eq}.

\begin{lemma} \label{surjection-lemma}
The natural map $V(K_\l)\rightarrow\Phi\bigl(J'_{/K_\l}\bigr)_{\!\m}$ is surjective.
\end{lemma}

\begin{proof} There is a commutative diagram
\[ \xymatrix{V(K_\l)\ar[r]\ar[dr]&E^2(K_\l)\times J'(K_\l)\ar[r]\ar[d]&J(K_\l)\ar[d]\\&\Phi\bigl(J'_{/K_\l}\bigr)_{\!\m}\ar[r]&\Phi\bigl(J_{/K_\l}\bigr)_{\!\m}} \]
and, by Lemma \ref{Phi-iso-lemma}, $\Phi(J)_\m$ is isomorphic to $\Phi(J_0^D(M\l))_\m$. But the group of components of $J_0^D(M\l)$ over an unramified extension of $\Q_\l$ is Eisenstein, hence the latter group is trivial. The claim then follows as in the proof of \cite[Lemma 3.6]{BD-ajm}. \end{proof} 

The study of component groups we are about to tackle is analogous to what is done in \cite{BD-ajm} in the case of classical modular curves. 

\begin{proposition} \label{iso-prop}
The map
\[ \boldsymbol i:E^2(K_\l)^\epsilon_\m\longrightarrow J(K_\l)^\epsilon_\m \]
induced by $\varphi_E$ is an isomorphism.
\end{proposition}

\begin{proof} By Lemma \ref{support-lemma}, we already know that $\boldsymbol i$ is injective. In order to prove surjectivity, consider the commutative diagram
\begin{equation} \label{diagram-eq}
\xymatrix{&0\ar[d]&0\ar[d]\\&E^2(K_\l)^\epsilon_\m\ar[d]\ar@{=}[r]&E^2(K_\l)^\epsilon_\m\ar[d]^-{\boldsymbol i}\\
V(K_\l)^\epsilon_\m\ar[r]\ar[dr]&\bigl(E^2(K_\l)\times J'(K_\l)\bigr)^{\!\epsilon}_{\!\m}\ar[d]\ar[r]^-{\varphi_\m}&J(K_\l)^\epsilon_\m\\
&\Phi\bigl(J'_{/K_\l}\bigr)^{\!\epsilon}_{\!\m}.&}
\end{equation}
First of all, observe that the leftmost vertical sequence is exact. Indeed, by Corollary \ref{mumford-coro}, the kernel of the map $J'(K_\l)^\epsilon\rightarrow\Phi(J'_{/K_\l})^\epsilon$ is isomorphic to an extension of a group of exponent $\l-1$ by a pro-$\l$ group; but $p\nmid\l(\l-1)$ because $p\not=\l$, $p\,|\,\l+1$ and $p\not=2$ (\S \ref{p-subsec} and \S \ref{l-subsec}), hence the support of this kernel is disjoint from $\m$.

Furthermore, the map $\varphi_\m$ is surjective. In fact, taking $K_\l$-rational points in the short exact sequence
\[ 0\longrightarrow V\longrightarrow E^2\times J'\overset\varphi\longrightarrow J\longrightarrow0 \]
gives a long exact sequence in cohomology
\[ E^2(K_\l)\times J'(K_\l)\longrightarrow J(K_\l)\longrightarrow H^1(K_\l,V)\longrightarrow H^1(K_\l,E^2)\times H^1(K_\l,J'). \]  
Since $\l\nmid pN$, the Galois representation $V$ (over $K_\l$) is unramified. Since the prime $\l$ is of good reduction for $E$, Lemma \ref{kernel-lemma} implies that the kernel of $H^1(K_\l,V)\rightarrow H^1(K_\l,E^2)$ is the finite part $H^1_f(K_\l,V)$ of the cohomology (recall that $H^1_f(K_\l,V)$ is, by definition, the kernel of the restriction map $H^1(K_\l,V)\rightarrow H^1(I_\l,V)$ where $I_\l$ is the inertia group at $\l$). Now let $V^0$ be the kernel of the natural map $V\rightarrow\Phi\bigl(J'_{/K_\l}\bigr)$. The kernel of $H^1_f(K_\l,V)\rightarrow H^1(K_\l,J')$ is contained in the image of $H^1_f(K_\l,V^0)$, hence the cokernel of $\varphi_\m$ is a quotient of $H_f^1(K_\l,V^0)^\epsilon_\m$. But complex conjugation acts as $-\epsilon$ on $(V^0)_\m$ and trivially on $\Gal(K_\l^{\text{nr}}/K_\l)$, so $H_f^1(K_\l,V^0)^\epsilon_\m=0$.

Finally, the diagonal map in diagram \eqref{diagram-eq} is surjective by Lemma \ref{surjection-lemma}, and the surjectivity of $\boldsymbol i$ can be checked via a diagram chasing. \end{proof}

Resume the notation of \S \ref{ring-class-subsec}, so that $\mathcal L$ is the ring class field of $K$ of conductor $\l$. Following \cite[p. 274]{BD-ajm}, let us introduce the group of local points
\[ \tilde J:=\frac{J(\mathcal L_\l)}{\varphi_E(E^2(\mathcal L_\l))+(\sigma-1)J(\mathcal L_\l)}. \]
The module $\tilde J$ is endowed with a Hecke action and with an action of complex conjugation. Since $\varphi_E(E^2)$ is contained in the identity component of $J$ and $\Gal(\mathcal L_\l/K_\l)$ acts trivially on $\Phi(J_{/\mathcal L_\l})$ (because it acts trivially on $\Phi(J_0^D(M\l)_{/\mathcal L_\l})$), projection onto the group of connected components gives a map
\[ \boldsymbol p:\tilde J_\m^\epsilon\longrightarrow\Phi(J_{/\mathcal L_\l})_\m. \] 

\begin{proposition} \label{p-iso-prop}
The map $\boldsymbol p$ is an isomorphism.
\end{proposition}

\begin{proof} The surjectivity of $\boldsymbol p$ is a direct consequence of its definition. To show injectivity, one can apply Lemma \ref{mumford-coro} and argue exactly as in the proof of \cite[Proposition 3.8]{BD-ajm}. \end{proof}

\begin{corollary} \label{p-coro}
The group $(\tilde J/\m\tilde J)^\epsilon$ is a one-dimensional vector space over $\T/\m$.
\end{corollary}

\begin{proof} The quotient $\Phi(J_{/\mathcal L_\l})/\m\Phi(J_{/\mathcal L_\l})$ is one-dimensional over $\T/\m$ by a combination of Theorem \ref{1-dimensional-thm} and Proposition \ref{non-eisenstein-prop}, and then the corollary follows from Proposition \ref{p-iso-prop}. \end{proof} 

Let $N_{\mathcal L_\l/K_\l}:=\sum_{i=1}^{(\l+1)/u}\sigma^i$ be the usual norm map of $\mathcal L_\l$ over $K_\l$. Since $(\l+1)/u$ is prime to $\l$ and the residual characteristic $p$ of $\m$ divides $\l+1$ (cf. \S \ref{l-subsec}), it follows that $N_{\mathcal L_\l/K_\l}$ induces a well-defined map
\[ \boldsymbol n:\tilde J/\m\tilde J\longrightarrow J(K_\l)/\m J(K_\l) \]
which is equivariant for the action of complex conjugation (see \cite[p. 275]{BD-ajm} for details). By an abuse of notation, we shall adopt the same symbol to denote the restriction of $\boldsymbol n$ to the $\epsilon$-eigenspaces. 

\begin{proposition} \label{n-iso-prop}
The map
\[ \boldsymbol n:(\tilde J/\m\tilde J)^\epsilon\longrightarrow\bigl(J(K_\l)/\m J(K_\l)\bigr)^{\!\epsilon} \]
is an isomorphism.
\end{proposition}

\begin{proof} To begin with, we prove the surjectivity of $\boldsymbol n$. To do this, consider the commutative diagram
\[ \xymatrix{&\bigl(J'(\mathcal L_\l)/\m J'(\mathcal L_\l)\bigr)^{\!\epsilon}\ar[r]\ar[d]&(\tilde J/\m\tilde J)^\epsilon\ar[d]^-{\boldsymbol n}\\
                     \bigl(V(K_\l)/\m V(K_\l)\bigr)^{\!\epsilon}\ar[r]\ar[dr]&\bigl((E^2(K_\l)\times J'(K_\l))/\m\bigr)^{\!\epsilon}\ar[d]\ar[r]^-{\varphi_\m}&\bigl(J(K_\l)/\m J(K_\l)\bigr)^{\!\epsilon}\\&\bigl(E^2(K_\l)/\m E^2(K_\l)\bigr)^{\!\epsilon},&} \]
where the topmost, left vertical arrow is induced by the norm map and the bottom vertical arrow is the natural projection onto the first component. As in the proof of \cite[Proposition 3.10]{BD-ajm}, we point out the following facts.
\begin{enumerate}
\item The leftmost vertical sequence is exact. Clearly, we need only check that the kernel of the second map is contained in the image of the first, which is true (since the extension $\mathcal L_\l/K_\l$ is abelian) by local class field theory. Indeed, Corollary \ref{mumford-coro} ensures that $J'(\mathcal L_\l)$ is isomorphic to a quotient of $(\mathcal L^\times_\l)^d$ by a discrete subgroup, and so the image of the norm map contains $J'(K_\l)^\epsilon$. 
\item The surjectivity of $\varphi_\m$ is a consequence of the surjectivity of the map denoted by the same symbol in the proof of Proposition \ref{iso-prop}.
\item The diagonal map is surjective by part (2) of Lemma \ref{kernel-lemma}.   
\end{enumerate}
Now the surjectivity of $\boldsymbol n$ follows by combining these three remarks with a diagram chasing.

Finally, to prove that $\boldsymbol n$ is injective observe that, by Corollary \ref{p-coro}, $(\tilde J/\m\tilde J)^\epsilon$ is a one-dimensional $\F_p$-vector space and that the same is true of $\bigl(J(K_\l)/\m J(K_\l)\bigr)^{\!\epsilon}$ by Proposition \ref{iso-prop} (cf. the proof of \cite[Lemma 3.4]{BD-ajm} for details). \end{proof}

By a slight abuse of notation, write
\[ \boldsymbol i:\bigl(E^2(K_\l)/\m E^2(K_\l)\bigr)^{\!\epsilon}\overset\cong\longrightarrow\bigl(J(K_\l)/\m J(K_\l)\bigr)^{\!\epsilon},\qquad\boldsymbol p:(\tilde J/\m\tilde J)^\epsilon\overset\cong\longrightarrow\Phi(J_{/\mathcal L_\l})/\m \]
for the isomorphisms of one-dimensional $\F_p$-vector spaces induced by the maps in Propositions \ref{iso-prop} and \ref{p-iso-prop}, respectively. Lemma \ref{Phi-iso-lemma} implies that there is an isomorphism
\[ \boldsymbol j:\Phi\bigl(J_0^D(M\l)_{/M_\l}\bigr)\big/\m\overset\cong\longrightarrow\Phi(J_{/\mathcal L_\l})/\m. \]
Since the maximal ideal $\m$ is not Eisenstein and $X^*(T)=\mathscr M^0$ canonically, Proposition \ref{non-eisenstein-prop} gives an identification $\Phi\bigl(J_0^D(M\l)_{/\mathcal L_\l}\bigr)\big/\m=\mathscr M^0/\m\mathscr M^0$. Then, as in \cite[p. 276]{BD-ajm}, define the isomorphism of one-dimensional $\F_p$-vector spaces
\begin{equation} \label{eta-eq}
\eta:=\boldsymbol j^{-1}\circ\boldsymbol p\circ\boldsymbol n^{-1}\circ\boldsymbol i:\bigl(E^2(K_\l)/\m E^2(K_\l)\bigr)^{\!\epsilon}\overset\cong\longrightarrow\mathscr M^0/\m\mathscr M^0. 
\end{equation}
It will play an important role in our subsequent arguments.

\section{Special values and Jochnowitz congruences}

\subsection{The module $\mathscr M^0_g$}

Recall from \S \ref{enhanced-subsec} that $\mathscr M=\Z[\E]$, the free abelian group over the (finite) set $\E$ of conjugacy classes of oriented Eichler orders of level $M$ in $\mathcal B$. Then the character group $X^*(T)$ can be identified with the subgroup $\mathscr M^0$ of degree zero divisors in $\mathscr M$. Of course, there is a short exact sequence
\begin{equation} \label{M-exact-eq}
0\longrightarrow\mathscr M^0\longrightarrow\mathscr M\xrightarrow\deg\Z\longrightarrow0
\end{equation}
induced by the usual degree map. The module $\mathscr M$ is endowed with a natural Hecke action compatible with the inclusion $\mathscr M^0\subset\mathscr M$. 

Notation being as in \S \ref{J-subsec}, if $g$ is a form on $J'$ then let $\cO_{g,\m}$ be the completion of $\cO_g$ at $\m_g$. Moreover, set
\[ \mathscr M^0_g:=(\mathscr M^0/I_g)\otimes_\T\cO_{g,\m},\qquad\mathscr M_g:=(\mathscr M/I_g)\otimes_\T\cO_{g,\m}. \]

\begin{proposition} \label{M-iso-prop}
The natural map $\mathscr M^0_g\rightarrow\mathscr M_g$ is an isomorphism and $\mathscr M^0_g$ is free of rank one over $\cO_{g,\m}$.
\end{proposition}

\begin{proof} As in \cite[Proposition 4.1]{BD-ajm}, the first part follows by tensoring exact sequence \eqref{M-exact-eq} with $\cO_{g,\m}$ and using the fact that the Hecke action on $\mathscr M/\mathscr M^0=\Z$ is Eisenstein, while the second part is a consequence of Lemma \ref{controllable-lem} plus Theorem \ref{1-dimensional-thm}. Details on this circle of ideas can be found in \cite[\S 2.4]{Vat1} and \cite[\S 2.1]{Vat2}. \end{proof}

\subsection{Algebraic parts of special values} \label{algebraic-part-subsec}

Recall that an embedding $\psi:K\hookrightarrow\mathcal B$ is an optimal embedding of $\cO_K$ into an Eichler order $R$ of $\mathcal B$ if $\psi^{-1}(R)=\cO_K$. Fix an orientation (in the sense of \cite[\S 2.2]{BD96}) of $\cO_K$ and let $h:=\#\Pic(\cO_K)$ be the class number of $K$. There are exactly $h$ distinct conjugacy classes $[\psi_1],\dots,[\psi_h]$ of oriented optimal embeddings of $\cO_K$ into some oriented Eichler order of $\mathcal B$, which correspond to Gross--Heegner points of conductor $1$ (see \cite[\S 2.6]{Vat1}); in fact, there is a simply transitive action of $\Pic(\cO_K)$ on the $[\psi_j]$ (cf. \cite[\S 1]{BD97} and the references therein; in particular, see \cite[\S 2.3]{BD96}, \cite[\S 3]{Gr} and \cite[\S 1.6]{Ro}). Each such conjugacy class $[\psi_j:\cO_K\hookrightarrow R_{\psi_j}]$ gives rise to the element $[R_{\psi_j}]\in\mathscr M$, and we define
\[ \psi_K:=[R_{\psi_1}]+\dots+[R_{\psi_h}]\in\mathscr M. \]
Note that the $h$ classes $[R_{\psi_j}]\in\E$ need not be distinct.

As in \cite[Definition 4.2]{BD-ajm}, we give the following

\begin{definition} \label{alg-part-dfn}
The \emph{algebraic part} $\LL_K(g,1)$ of $L_K(g,1)$ is the image of $\psi_K$ in the rank one $\cO_{g,\m}$-module $\mathscr M^0_g=\mathscr M_g$.
\end{definition}

This definition is justified by the next result, whose proof proceeds along lines similar to those of \cite[Theorem 4.3]{BD-ajm}.

\begin{theorem} \label{alg-part-thm}
$L_K(g,1)=0$ if and only if $\LL_K(g,1)=0$.
\end{theorem}

\begin{proof} View $\psi_K$ as an element of $\mathscr M\otimes\C$ and write $\psi_{K,g}$ for the projection of $\psi_K$ on the $g$-isotypic component of $\mathscr M\otimes\C$. Then
\begin{equation} \label{iff-eq}
\psi_{K,g}=0\;\Longleftrightarrow\;\LL_K(g,1)=0. 
\end{equation}
By multiplicity one, the $g$-isotypic component of $\mathscr M\otimes\C$ is a one-dimensional $\C$-vector space (cf. \cite[\S 2.1]{Vat2}) and the pairing $\langle\,,\rangle$ on $\mathscr M$ defined in \eqref{scalar-product-eq} induces a perfect, nondegenerate pairing on it. By \cite[Theorem 1.1]{BD97}, one has the formula
\begin{equation} \label{gross-eq}
\frac{L_K(g,1)}{(g,g)}=\frac{\langle\psi_{K,g},\psi_{K,g}\rangle}{u^2\sqrt{\delta_K}}, 
\end{equation}
where $(g,g)$ is the Petersson scalar product of $g$ with itself, $u$ is equal to $\#\cO_K^\times/2$ and, as before, $\delta_K$ is the discriminant of $K$ (see \cite[\S 2.8]{Vat1} for a more general formulation). The result follows by comparing \eqref{iff-eq} and \eqref{gross-eq}. \end{proof}

\begin{remark}
A special case of formula \eqref{gross-eq} was first proved in \cite{Gr} by Gross, whose work was generalised by Daghigh in \cite{Daghigh}. A formula for Hilbert cusp forms of parallel weight $2$ twisted by finite characters was finally obtained by Zhang in \cite[Theorem 1.3.2]{Zh2}.
\end{remark} 

\subsection{Review of Heegner points} \label{heegner-subsec}

In this subsection we briefly review some basic results about Heegner points on (indefinite) Shimura curves; for more details, see \cite[Section 2]{BD96}.

Suppose that we have chosen orientations of $\cO_K$ and of $R(M)$ (see \cite[\S 1.1 and \S 2.2]{BD96}). For our purposes, a \emph{Heegner point} $P$ of conductor $1$ on $X_0^D(M)$ is the image in $X_0^D(M)$ of the fixed point in $\mathcal H$ for the action of $\C^\times$ via an oriented optimal embedding of $\cO_K$ into $R(M)$ and isomorphism $i_\infty$ of \eqref{indefinite-iso-eq}. Complex multiplication ensures that $P\in X_0^D(M)(H)$. If, as before, $h$ is the class number of $K$, there are precisely $h$ Heegner points $P_1,\dots,P_h$ of conductor $1$ on $X_0^D(M)$, which are permuted transitively by $\Gal(H/K)$ (see \cite[\S 2.3]{BD96}).

Analogously, there are $h$ Heegner points $\mathcal P_1,\dots,\mathcal P_h$ of conductor $1$ on $X_0^D(M\l)$. Replacing $\cO_K$ with the order of $K$ of conductor $\l$, one can also consider Heegner points of conductor $\l$ on $X_0^D(M\l)$, which are rational over $\mathcal L$. Since $\l$ is inert in $K$, there are exactly $h(\l+1)/u$ such points, to be denoted $\mathcal P'_i$, on which $\Gal(\mathcal L/K)$ acts simply transitively (\cite[Lemma 2.5]{BD96}).

Let $N_{\mathcal L/H}$ be the norm map. As explained in \cite[\S 2.4]{BD96}, we can (and do) choose orientations of the quadratic orders and of the Eichler orders in such a way that (up to renumbering) for every $i=1,\dots,h$ there is an equality   
\[ uN_{\mathcal L/H}(\mathcal P'_i)=\mathcal P_i \]
of divisors on $X_0^D(M\l)$. Now define
\[ P_K:=\sum_{i=1}^h\iota_M(P_i)\in J_0^D(M)(K),\qquad\mathcal P'_{\mathcal L}:=\sum_{i=1}^h\iota_{M\l}(\mathcal P'_i)\in J_0^D(M\l)(\mathcal L). \]
Having the points $\mathcal P_i$ at our disposal, fix the orientation of $R(M)$ so that
\[ \pi_1(\mathcal P_i)=P_i \]
for all $i=1,\dots,h$. Thanks to the commutativity of square \eqref{iota-square-eq}, it follows that
\begin{equation} \label{norm-heegner-eq}
P_K=u\pi_{1,*}\bigl(N_{\mathcal L/H}(\mathcal P'_{\mathcal L})\bigr). 
\end{equation}
Finally, define the Heegner point
\[ \alpha_K:=\pi_{E,*}(P_K)\in E(K). \]
This is the point in terms of which we shall state our main result.

\subsection{Jochnowitz congruences and main result}

Recall the isomorphism
\[ \eta:\bigl(E^2(K_\l)/\m E^2(K_\l)\bigr)^{\!\epsilon}\overset\cong\longrightarrow\mathscr M^0/\m\mathscr M^0 \]
defined in \eqref{eta-eq} and let $\tilde\pi_J:J_0^D(M\l)(M_\l)\rightarrow\tilde J/\m\tilde J$ be the map induced by $\pi_J$. Now that we have all the ingredients at hand, the proof of the main result of this paper (Theorem \ref{main-thm}) follows that of \cite[Theorem 6.1]{BD-ajm} closely. First of all, we need two lemmas.

\begin{lemma} \label{deg-lemma}
$u\deg(\pi_E)\cdot\boldsymbol n\bigl(\tilde\pi_J(\mathcal P'_{\mathcal L})\bigr)=\boldsymbol i\bigl((\alpha_K,0)\bigr)$.
\end{lemma}

\begin{proof} Adopt the symbol $[\star]$ to denote the class modulo $\m$ of an element $\star$. Using \eqref{norm-heegner-eq}, the functoriality of the Hodge class (cf. \cite[eq. (3.7)]{CV}) and the fact that the Hecke action on Hodge classes is Eisenstein, it can be checked that the equality
\[ u\cdot\boldsymbol n\bigl(\tilde\pi_J(\mathcal P'_{\mathcal L})\bigr)=\bigl[\pi_J\bigl(\pi_1^*(P_K)\bigr)\bigr] \]
holds in $J(K_\l)/\m J(K_\l)$. On the other hand, we know that
\[ \deg(\pi_E)P_K=\pi_E^*(\alpha_K) \]
in $J(K_\l)/\m J(K_\l)$. Therefore 
\[ u\deg(\pi_E)\cdot\boldsymbol n\bigl(\tilde\pi_J(\mathcal P'_{\mathcal L})\bigr)=\pi_J\bigl(\pi_1^*(\pi_E^*(\alpha_K))\bigr)=\pi_J\Big(\pi^*\bigl(\pi_E^*((\alpha_K,0))\bigr)\!\Big)=\boldsymbol i\bigl((\alpha_K,0)\bigr), \]
with the last equality following from the definition of $\boldsymbol i$ (cf. Proposition \ref{iso-prop}). \end{proof}

\begin{lemma} \label{lemma2}
$\boldsymbol p\bigl(\tilde\pi_J(\mathcal P'_{\mathcal L})\bigr)=\boldsymbol j\bigl(\LL_K(g,1)\bigr)$.
\end{lemma}

\begin{proof} This is a consequence of \cite[Theorem 3.2]{BD97}. \end{proof}

Write $[\LL_K(g,1)]$ for the class of $\LL_K(g,1)$ in the quotient $\mathscr M^0_g/\m_g\mathscr M^0_g=\mathscr M^0/\m\mathscr M^0$. 

\begin{theorem} \label{main-thm}
If $g$ is an eigenform on $J'$ then the equality
\[ \eta\bigl((\alpha_K,0)\bigr)=u\deg(\pi_E)\cdot[\LL_K(g,1)] \]
holds in $\mathscr M^0/\m\mathscr M^0$.
\end{theorem}

\begin{proof} Let $J_0^D(M\l)(\mathcal L_\l)^\epsilon$ denote the submodule of points in $J_0^D(M\l)$ whose natural image in the quotient $J_0^D(M\l)(\mathcal L_\l)\big/(\sigma-1)J_0^D(M\l)(\mathcal L_\l)$ belongs to the $\epsilon$-eigenspace for the action of complex conjugation. There is a diagram
\begin{equation} \label{final-diagram-eq}
\xymatrix{&J_0^D(M\l)(\mathcal L_\l)^\epsilon\ar[d]^-{\tilde\pi_J}\ar[r]&\mathscr M^0/\m\mathscr M^0\ar[d]^-{\boldsymbol j}\\
                 &(\tilde J/\m\tilde J)^\epsilon\ar[d]^-{\boldsymbol n}\ar[r]^-{\boldsymbol p}&\Phi(J_{/\mathcal L_\l})/\m\\
                 \bigl(E^2(K_\l)/\m E^2(K_\l)\bigr)^{\!\epsilon}\ar[r]^-{\boldsymbol i}&\bigl(J(K_\l)/\m J(K_\l)\bigr)^{\!\epsilon},&} 
\end{equation}
and the theorem follows by combining Lemmas \ref{deg-lemma} and \ref{lemma2} with the commutativity of \eqref{final-diagram-eq}. \end{proof}

Theorem \ref{main-thm} immediately implies Theorem \ref{main-intro-thm}. 

\begin{remark}
By the formula of Gross--Zagier type proved by Zhang in \cite[Theorem C]{Zh}, the Heegner point $\alpha_K$ encodes, via its N\'eron--Tate height, the central derivative $L'_K(E,1)$, hence Theorem \ref{main-thm} (or, rather, Theorem \ref{main-intro-thm}) can be viewed as providing a congruence modulo $\m$ between $L'_K(E,1)$ and $L_K(g,1)$. 
\end{remark}

\subsection{Some consequences} \label{consequences-subsec}

Here we collect two consequences of our main result.

\begin{proposition} \label{J'-prop}
If the image of $\alpha_K$ in $E(K_\l)/pE(K_\l)$ is nonzero then $L_K(J',1)\not=0$.
\end{proposition}

\begin{proof} By Theorem \ref{main-intro-thm}, $[\LL_K(g,1)]\not=0$ for all eigenforms $g$ on $J'$. In particular, $\LL_K(g,1)\not=0$, hence $L_K(g,1)\not=0$ by Theorem \ref{alg-part-thm}. But $L_K(J',1)=\prod_g L_K(g,1)$, where the product is taken over the distinct eigenforms on $J'$, and the claim follows. \end{proof}

\begin{proposition} \label{m-selmer-prop}
If the image of $\alpha_K$ in $E(K_\l)/pE(K_\l)$ is nonzero then
\begin{enumerate}
\item the $p$-Selmer group of $E$ over $K$ is one-dimensional over $\F_p$ and is generated by the image of $\alpha_K$ under the connecting homomorphism of Kummer theory;
\item the $\m$-Selmer group of $J'$ is trivial, hence $J'(K)$ is finite. 
\end{enumerate}
\end{proposition}

\begin{proof}[Sketch of proof.] The assumption implies that $\alpha_K$ is not divisible by $p$ in $E(K)$, so part (1) follows from a theorem of Kolyvagin (see, e.g., \cite[Theorem 10.2]{Darmon} and \cite[Proposition 2.1]{Gr2}). As for part (2), it follows from Theorem \ref{main-intro-thm} and the natural generalisation of the results in \cite{BD97} to the case of eigenforms with not necessarily rational Hecke eigenvalues. \end{proof}

\begin{remark}
Strictly speaking, Kolyvagin proved his theorem in the case where the point $\alpha_K$ comes via a modular parametrization $X_0(N)\rightarrow E$ from a Heegner point on the classical modular curve $X_0(N)$. However, since Heegner points on modular curves and Heegner points on Shimura curves enjoy the same formal properties, his arguments carry over \emph{verbatim} to our more general situation.
\end{remark} 

For arithmetic results in the same vein as Proposition \ref{m-selmer-prop}, the reader is referred to the paper \cite{GP} by Gross and Parson.

\end{document}